\documentclass[aap]{imsart}

\RequirePackage{amsthm,amsmath,amsfonts,amssymb,mathtools,comment,color,enumitem}
\RequirePackage[authoryear]{natbib}
\usepackage{xcolor}
\usepackage[
    colorlinks=true,
    linkcolor=blue,
    citecolor=blue,
    urlcolor=blue
]{hyperref}

\startlocaldefs

\newcommand\abs[1]{\left\lvert#1\right\rvert}
\newcommand\prs[1]{\left(#1\right)}
\newcommand\sbk[1]{\left[#1\right]}

\newcommand{\gti}{\rightarrow\infty}

\newcommand{\R}{\mathbb{R}}

\newcommand{\E}{\mathbb{E}}

\theoremstyle{plain}
\newtheorem{theorem}{Theorem}
\newtheorem{lemma}{Lemma}
\newtheorem{proposition}{Proposition}
\newtheorem{corollary}{Corollary}
\theoremstyle{remark}
\newtheorem{definition}{Definition}

\endlocaldefs

\begin{document}

\begin{frontmatter}
\title{A Harris recurrent continuous-time Markov process without wide-sense regenerative structure}
\runtitle{Harris recurrence without wide-sense regeneration}

\begin{aug}
\author[1]{\fnms{Yanlin} \snm{Qu}\ead[label=e1,mark]{quyanlin@cuhk.edu.cn}}
\and
\author[2]{\fnms{Peter} \snm{Glynn}\ead[label=e2,mark]{glynn@stanford.edu}}
\address[1]{School of Data Science, The Chinese University of Hong Kong, Shenzhen}
\address[2]{Department of Management Science and Engineering, Stanford University}
\address{\printead{e1,e2}}
\end{aug}

\begin{abstract}
While Harris recurrent Markov chains (in discrete time) automatically exhibit wide-sense regenerative structure, we construct a Harris recurrent Markov process (in continuous time) that is not wide-sense regenerative, thereby giving a negative answer to the open problem first raised in the 1990s and later posed in \cite{glynn2011wide}.
The counterexample exhibits the following rigidity property: every almost surely finite random time that is independent of the state observed at that time must be almost surely constant. A Cantor set linearly independent over the rationals plays a key role in the construction, turning calendar time into an algebraic record of the path already traversed.
\end{abstract}



\end{frontmatter}

\section{Introduction}
Harris recurrence \citep{harris1956existence} provides a useful tool for the long-run analysis of Markov models, supporting the existence of invariant measures, laws of large numbers, and central limit theorems. In discrete time, Harris recurrent Markov chains automatically exhibit wide-sense regenerative structure \citep{athreya1978new,nummelin1978splitting}, bringing the full arsenal of renewal theory to bear on questions of convergence to equilibrium. In continuous time, however, whether Harris recurrent Markov processes are automatically wide-sense regenerative has long remained open, as described in \citet{glynn2011wide}. The question is not only of theoretical interest but also of practical importance. In particular, when Markov processes are specified through their infinitesimal generators, establishing Harris recurrence is often substantially easier than verifying sufficient conditions for wide-sense regeneration.

We answer the question in the negative by constructing a Harris recurrent Markov process that is not wide-sense regenerative. In fact, the counterexample satisfies a stronger ``rigidity'' property: every almost surely finite random time that is independent of the state observed at that time must be almost surely constant, which provides the basic obstruction to wide-sense regeneration.

\textbf{From recurrence to regeneration.} Let $X=(X(t):t\geq 0)$ be a (time-homogeneous) strong Markov process on a separable metric space $\mathcal{S}$ (endowed with its Borel $\sigma$-algebra $\mathcal{B}$), with right-continuous paths having left limits, and let $P_x$ denote its law when $X(0)=x$. For $t\geq0$, let $\Theta_t X=(X(t+s):s\geq0)$ denote the process shifted by $t$.
\begin{definition}[Harris recurrence]
\label{def:harris}
The process $X$ is said to be \textit{Harris recurrent} if there exists a non-trivial $\sigma$-finite measure $\eta$ on $\mathcal{S}$ for which $\eta(A)>0$ implies that
\[
\int_0^\infty I(X(t)\in A)dt=\infty,\;\;\;\;P_x\text{-a.s.},
\]
for each $x\in\mathcal{S}$.
\end{definition}
\begin{definition}[Wide-sense regeneration]
\label{def:regeneration}
The process $X$ is said to be \textit{wide-sense regenerative} if there exist randomized stopping times $0\leq R_0<R_1<\dots<\infty$ such that:
\begin{enumerate}[label=(\roman*)]
    \item for each $n\geq0$, $\Theta_{R_n}X$ is independent of $R_n$;
    \item $\Theta_{R_n}X$ is identically distributed in $n$.
\end{enumerate}
\end{definition}
A standard route from recurrence to regeneration is through a minorization condition. Suppose there exist a set $K\in\mathcal{B}$ with $\eta(K)>0$, constants $h,\lambda>0$, and a probability measure $\phi$ on $\mathcal{S}$ such that
\[
P_x(X(h)\in B)\geq\lambda\phi(B),\;\;\;\;x\in K,\;\;B\in\mathcal{B}.
\]
Whenever $X$ visits $K$, with probability $\lambda$, it ``regenerates'' after $h$ units of time according to the distribution $\phi$, so recurrent visits lead to a sequence of regeneration times. In discrete time, the required minorization can be established for Harris recurrent Markov chains via a measure-differentiation argument; see, e.g., pp.~103--105 of \cite{meyn2012markov}. However, this argument can fail in continuous time; see \cite{glynn2011wide} for an example illustrating this subtlety.

\textbf{Counterexample in continuous time.}
We are now ready to state our main result, which gives a negative answer to the open problem:
\[
    \textit{Does every Harris recurrent Markov process necessarily exhibit wide-sense regeneration?}
\]
\begin{theorem}
\label{thm_main}There exist a compact separable metric space $\mathcal{S}$ and a time-homogeneous strong Markov process $X$ on $\mathcal{S}$, with right-continuous paths having left limits, such that $X$ is Harris recurrent but, for every initial state $x\in\mathcal{S}$, $X$ under $P_x$ is not wide-sense regenerative.
\end{theorem}
This $X$ can be viewed as a process moving on an uncountable directed graph. The vertex space is $\mathcal{V}=\{0,1\}^\mathbb{N}$, the space of infinite binary sequences, with a directed edge from $u$ to $v$ for every $u,v\in \mathcal{V}$. Each directed edge $(u,v)$ is assigned a distinct length $l(u,v)\in[1,2]$ from a Cantor set linearly independent over the rationals\footnote{The existence of perfect sets with stronger algebraic independence properties goes back to \cite{neumann1928system}.}, i.e., no nontrivial rational linear combination of finitely many distinct edge lengths can vanish. The process moves along the edges at unit speed. Upon reaching the next vertex, it samples a new vertex in an iid manner (i.e., the vertex sequence is iid), and starts moving along the corresponding new edge.
At any time, the state records the edge currently being traversed and the position of the process along that edge, giving a Markovian description of the dynamics. The iid vertex sequence makes the Markov process Harris recurrent, while the linearly independent (over the rationals) edge lengths make calendar time retain algebraic information about the path already traversed, thereby preventing wide-sense regeneration.

The remainder of the paper is organized as follows. In Section \ref{sec_counter}, we construct the counterexample in detail. In Section \ref{sec_harris}, we show that the process is Harris recurrent. In Section \ref{sec_regen}, we show that the process is not wide-sense regenerative. In Section \ref{sec_cantor}, for completeness, we provide an elementary construction of the Cantor set used in the counterexample. In Section \ref{sec_details}, we collect the remaining technical details.

\textbf{Use of AI.}
This work was developed in collaboration with GPT-5.6 Sol Pro. In particular, the authors thank GPT for bringing the rationally independent Cantor set to their attention, which lit the way to resolving the open problem. Starting from GPT's attempted resolution, the authors checked and corrected the arguments line by line, taking responsibility for all claims in the paper.

\section{Counterexample Construction}
\label{sec_counter}
Let \[\mathcal{V}=\{0,1\}^\mathbb{N}\] be endowed with the product topology.
Note that $\mathcal{V}^2$ is homeomorphic to $\mathcal{V}$ via coordinate interleaving
\[
\mathcal{V}^2\rightarrow \mathcal{V},\;\;\;\;((u_1,u_2,\dots),(v_1,v_2,\dots))\mapsto(u_1,v_1,u_2,v_2,\dots).
\]
Let \[\mu=\bigotimes_{k\geq1}\frac{\delta_0+\delta_1}{2}\]
be the probability measure under which the binary coordinates are iid and equally likely to be zero or one. 
Note that $\mu$ is atomless as the probability of any prescribed infinite binary sequence is zero. 

Starting from $[1,2]$, we construct the standard middle-third Cantor set $C$ by repeatedly removing the open middle third of each remaining interval. By Theorem 1 of \cite{mycielski1964independent}, there exists a Cantor subset $L\subset C$ that is linearly independent over the rationals.
That is, for every $k\geq1$, every collection of distinct points
$r_1,\dots,r_k\in L$, and every $q_1,\dots,q_k\in\mathbb Q$,
\[
    \sum_{i=1}^k q_i r_i=0
    \;\;\Rightarrow\;\;
    q_1=\dots=q_k=0.
\]
For completeness, we provide an elementary construction of such a set in Section \ref{sec_cantor}.

Since $\mathcal{V}^2$ and $L$ are both homeomorphic to the standard Cantor set, there exists a homeomorphism
\[
l: \mathcal{V}^2\rightarrow L,
\]
e.g., by selecting (left or right) intervals according to the interleaved binary sequence.
For $u,v\in \mathcal{V}$, we assign $l(u,v)\in[1,2]$ as the length of the directed edge from $u$ to $v$. Since $l$ is one-to-one, distinct directed edges have distinct lengths. Most importantly, no nontrivial rational linear combination of finitely many distinct edge lengths can vanish.

With the graph in place, we now describe the dynamics of the counterexample.
This process moves along the edges at unit speed. Upon reaching the next vertex, it samples a new vertex according to $\mu$, independently of the past, and starts moving along the corresponding new edge.
To give a Markovian description of the dynamics, the state records the edge currently being traversed and the (normalized) position of the process along that edge.
Let $\mathbb{T}=\mathbb{R}/\mathbb{Z}$ be the unit circle. For $z\in\mathbb{T}$, let $\bar{z}\in[0,1)$ be its canonical representative, and write $0_\mathbb{T}$ for the equivalence class of $0$.
Take
\[
\mathcal{S}=\mathcal{V}^2\times\mathbb{T}
\]
as the state space. 
Since both $\mathcal{V}$ and $\mathbb{T}$ are compact separable metric spaces, so is $\mathcal{S}$.
A state $x=(u,v,z)\in\mathcal{S}$ means that the process is traversing the edge $(u,v)$, with a fraction $\bar{z}$ of the edge already traversed. Thus the remaining time on the current edge is
\[
r(x)=(1-\bar{z})l(u,v).
\]
Starting from $x=(u,v,z)$, for $t\in[0,r(x))$, the process evolves deterministically according to
\[
X(t)=\prs{u,v,\sbk{\bar z+\frac{t}{l(u,v)}}_{\mathbb T}},
\]
where $[\cdot]_\mathbb{T}$ denotes reduction modulo $1$. At time $r(x)$, the process jumps to 
\[
X(r(x))=(v,w,0_\mathbb{T}),
\]
where $w$ is sampled from $\mu$, independently of the past. The same rule is then repeated.

The above construction defines a time-homogeneous Markov process with c\`adl\`ag sample paths. 
In fact, the process is strong Markov, the verification of which is deferred to the end of Section \ref{sec_details}.
The process is nonexplosive: after the first transition, every edge traversal takes at least one unit of time, so only finitely many transitions can occur on any bounded time interval.

\section{Harris Recurrence}
\label{sec_harris}
We now verify that the process constructed in Section \ref{sec_counter} is Harris recurrent. Define a measure $\eta$ on $\mathcal{S}$ by
\[
\eta(A)=\int_\mathcal{V}\int_\mathcal{V} g_A(u,v)\mu(du)\mu(dv),\;\;\;\;A\in\mathcal{B},
\]
where
\[
g_A(u,v)=\int_0^{l(u,v)}I\prs{\prs{u,v,\sbk{\frac{t}{l(u,v)}}_\mathbb{T}}\in A}dt.
\]
Thus, $\eta$ is the expected occupation measure during a complete traversal of an edge where the two endpoints are sampled independently from $\mu$. It is a non-trivial finite measure with
\[
\eta(\mathcal{S})=\int_\mathcal{V}\int_\mathcal{V} l(u,v)\mu(du)\mu(dv)\in[1,2].
\]
\begin{proposition}
\label{prop_harris}
The process $X$ is Harris recurrent with respect to $\eta$.
\end{proposition}
\begin{proof}
Fix $A\in\mathcal{B}$ with $\eta(A)>0$. Fix an arbitrary initial state $x=(u,v,z)\in\mathcal{S}$, and let $V_1,V_2,\dots\stackrel{\mathrm{iid}}{\sim}\mu$ be the successive vertices sampled by the process after it completes the initial edge, with $V_0=v$. Note that $(g_A(V_{2k-1},V_{2k}):k\geq1)$ is an iid sequence of random variables. By the strong law of large numbers,
\[
\frac{1}{n}\sum_{k=1}^n g_A(V_{2k-1},V_{2k})\rightarrow\int_\mathcal{V}\int_\mathcal{V} g_A(a,b)\mu(da)\mu(db)=\eta(A)>0,\;\;\;\;P_x\text{-a.s.},
\]
as $n\gti$.
Consequently,
\[
\begin{aligned}
\int_0^\infty I(X(t)\in A)dt\geq&\sum_{k=1}^\infty\int_{0}^{l(V_{2k-1},V_{2k})}I\prs{\prs{V_{2k-1},V_{2k},\sbk{\frac{t}{l(V_{2k-1},V_{2k})}}_\mathbb{T}}\in A}dt\\
=&\sum_{k=1}^\infty g_A(V_{2k-1},V_{2k})\\
=&\infty,\;\;\;\;P_x\text{-a.s.}.
\end{aligned}
\]
\end{proof}
\section{No Wide-sense Regeneration}
\label{sec_regen}
We now show that the process constructed in Section \ref{sec_counter} is not wide-sense regenerative. In fact, it satisfies a stronger ``rigidity'' property: every almost surely finite random time that is independent of the state observed at that time must be almost surely constant. To establish this property, we first prepare three ingredients.

The following lemma shows that the difference between the lengths of two simple paths with common endpoints can be used to recover a common non-initial vertex.
\begin{lemma}[Algebraic decoder]
\label{lem_decode}
There exists a Borel measurable function $c:\R\rightarrow\mathcal{V}$ with the following property. Let
\[
p=(p_0,p_1,\dots,p_m),\;\;\;\;q=(q_0,q_1,\dots,q_n)
\]
be two finite simple directed paths (no vertex is visited more than once) with the same initial vertex ($p_0=q_0$) and the same terminal vertex ($p_m=q_n$). If
\[
d=\sum_{i=1}^ml(p_{i-1},p_i)-\sum_{i=1}^nl(q_{i-1},q_i)\neq0,
\]
then $c(d)$ is a non-initial vertex occurring in both $p$ and $q$, i.e., 
\[
c(d)\in\{p_1,\dots,p_m\}\cap\{q_1,\dots,q_n\}.
\]
\end{lemma}
\begin{proof}
For a finite simple directed path $p$, let
\[
\mathcal{E}(p)=\{(p_{i-1},p_i):1\leq i\leq m\}
\]
be the finite set of edges traversed by $p$. Given $p$ and $q$, set
\[
A_+=\mathcal{E}(p)\setminus\mathcal{E}(q),\;\;\;\;A_-=\mathcal{E}(q)\setminus\mathcal{E}(p).
\]
After canceling the edges common to the two paths,
\[
d=d(A_+,A_-)=\sum_{e\in A_+}l(e)-\sum_{e\in A_-}l(e).
\]
We claim that $d$ uniquely determines $A_+$ and $A_-$. Indeed, suppose another pair of disjoint finite sets of directed edges $B_+$ and $B_-$ satisfies
\[
d(A_+,A_-)=d(B_+,B_-).
\]
Then
\[
\sum_e[I(e\in A_+)-I(e\in A_-)-I(e\in B_+)+I(e\in B_-)]l(e)=0,
\]
where the sum is over the finitely many edges appearing in the four sets. Since the distinct edge lengths are linearly independent over the rationals, every coefficient must vanish. Hence, for every directed edge $e$,
\[
I(e\in A_+)-I(e\in A_-)=I(e\in B_+)-I(e\in B_-).
\]
Since $A_+\cap A_-=\emptyset$ and $B_+\cap B_-=\emptyset$, the positive and negative supports of the two sides are respectively $A_+,A_-$ and $B_+,B_-$. Therefore,
\[
A_+=B_+,\;\;\;\;A_-=B_-.
\]
Thus the scalar $d$ uniquely determines the two signed edge sets $A_+$ and $A_-$. 

Next, define
\[
C(A_+,A_-)=\{v\in\mathcal{V}:(u,v)\in A_+\text{ and }(u',v)\in A_-\text{ for some }u,u'\in\mathcal{V}\},
\]
which is the set of vertices receiving both a positive and a negative incoming edge.
We claim that $C(A_+,A_-)\neq\emptyset$ when $p_0=q_0$, $p_m=q_n$, and $d\neq0$.
Since $d\neq0$, the paths are different. Follow the two paths from their common initial vertex until they first diverge, and then let $v$ be their first common vertex after the divergence, which is a non-initial vertex occurring in both paths.
Such a vertex exists because the two paths have the same terminal vertex. Note that they must enter $v$ along different directed edges. Since the paths are simple, the incoming edge used by $p$ cannot occur elsewhere in $q$, and vice versa. Hence one belongs to $A_+$ and the other belongs to $A_-$, so
\[
v\in C(A_+,A_-).
\]

Since $d$ uniquely determines $A_+$, $A_-$, and hence $C(A_+,A_-)$, define $c(d)$ to be the lexicographically smallest element of $C(A_+,A_-)\subset\{0,1\}^\mathbb{N}$ whenever possible\footnote{That is, when $d=d(A_+,A_-)$ for some $A_+$ and $A_-$ with $C(A_+,A_-)\neq\emptyset$.}, and set $c(d)$ to be some fixed element of $\mathcal{V}$ otherwise. The Borel measurability of $c$ is verified in Section \ref{sec_details}.
\end{proof}

The following lemma shows that a vertex independent of the conditioning state cannot occur in two conditionally independent finite vertex sequences (when no fixed vertex has positive probability of occurring).
\begin{lemma}[Conditional collision]
\label{lem_collision}
Let $Y$ be an $\mathcal{S}$-valued random variable, and let
\[
U=(U_1,\dots,U_N)
\]
be a random sequence of vertices in $\mathcal{V}$, where $N<\infty$ almost surely.
Conditional on $Y$, let $U^{(1)}$ and $U^{(2)}$ be independent copies of $U$. Suppose that\[
P(v\text{ occurs in }U)=0
\]
for every fixed $v\in\mathcal{V}$. If $Z$ is a $\mathcal{V}$-valued random variable independent of $Y$, then
\[
P(Z\text{ occurs in both }U^{(1)}\text{ and }U^{(2)})=0.
\]
No independence between $Z$ and the two sequences is required.
\end{lemma}
\begin{proof}
Adjoin an isolated cemetery point $\partial\notin\mathcal{V}$ and pad $U$ by setting $U_j=\partial$ for $j>N$. Let $K_j(y,\cdot)=P(U_j\in\cdot|Y=y)$, and define
\[
\mathcal{A}=\{(y,v)\in\mathcal{S}\times\mathcal{V}:K_j(y,\{v\})>0\text{ for some }j\geq1\}.
\]
The set $\mathcal{A}$ is Borel, which is verified in Section \ref{sec_details}. Write $\mathcal{A}_y=\{v:(y,v)\in\mathcal{A}\}$. For every fixed $v\in\mathcal{V}$ and every $j\geq1$,
\[
\E[K_j(Y,\{v\})]=P(U_j=v)=0.
\]
Therefore, $K_j(Y,\{v\})=0$ almost surely, and a countable union gives 
\[
P(v\in\mathcal{A}_Y)=P((Y,v)\in\mathcal{A})\leq\sum_{j=1}^\infty P(K_j(Y,\{v\})>0)=0
\]
for every fixed $v\in\mathcal{V}$. Since $Z$ is independent of $Y$,
\[
P(Z\in\mathcal{A}_Y)=\int_\mathcal{V}P(v\in \mathcal{A}_Y)P(Z\in dv)=0.
\]
Let $U_i^{(1)}$ and $U_j^{(2)}$ denote the padded coordinates of the two conditional copies. Conditional on $Y=y$, for every $i,j\geq1$,
\[
P(U_i^{(1)}=U_j^{(2)},U_i^{(1)}\in\mathcal{V}\setminus\mathcal{A}_y|Y=y)=\int_{\mathcal{V}\setminus\mathcal{A}_y}K_j(y,\{v\})K_i(y,dv)=0,
\]
where the last equality follows because 
\[v\in\mathcal{V}\setminus\mathcal{A}_y\Rightarrow (y,v)\notin\mathcal{A}\Rightarrow K_j(y,\{v\})=0\text{ for all }j\geq1.
\]
Taking a countable union over $i,j$, every common vertex of $U^{(1)}$ and $U^{(2)}$ belongs to $\mathcal{A}_Y$ almost surely, so
\[
P(Z\text{ occurs in both }U^{(1)}\text{ and }U^{(2)})\leq P(Z\in\mathcal{A}_Y)=0.
\]
\end{proof}

The following lemma records an elementary consequence of independence that will be useful in handling the initial part of the trajectory.
\begin{lemma}[Independence consequence]
\label{lem_independence}
Let $W$ be a real-valued random variable independent of $Y$, let $B\subset\mathcal{S}$ be Borel with $P(Y\in B)>0$, and let $h:\mathcal{S}\rightarrow\R$ be Borel. If $W=h(Y)$ on $\{Y\in B\}$, then $W$ is almost surely constant.
\end{lemma}
\begin{proof}
Let $\lambda(\cdot)=P(Y\in\cdot)$.
Note that
\[
\begin{aligned}
P(Y\in B)=&P(Y\in B,W=h(Y))\\
=&\int_B P(W=h(Y)|Y=y)\lambda(dy)\\
=&\int_B P(W=h(y)|Y=y)\lambda(dy)\\
=&\int_B P(W=h(y))\lambda(dy)\\
\leq&\int_B \lambda(dy)\\
=&P(Y\in B).
\end{aligned}
\]
Since $P(Y\in B)>0$, the equality above forces the integrand to attain its maximum value $1$ for $\lambda$-almost every $y\in B$. Thus, there is at least one $\tilde{y}\in B$ such that $P(W=h(\tilde{y}))=1$, making $W$ almost surely constant.
\end{proof}

With the three ingredients in place, we are ready to establish the rigidity property announced at the beginning of the section. For $y=(u,v,z)$, define
\[
\gamma(y)=(u,v),\;\;\;\;\beta(y)=u,\;\;\;\;\alpha(y)=\bar{z}l(u,v).
\]
Thus $\gamma(y)$ is the current edge, $\beta(y)$ its initial vertex, and $\alpha(y)$ the elapsed time on that edge.

\begin{proposition}[Time rigidity]
\label{prop_rigidity}
Fix $x\in\mathcal{S}$. On any randomized extension preserving the law of $X$ under $P_x$, let $T$ be an almost surely finite random time. If $T$ is independent of $X(T)$, then $T$ is almost surely constant.
\end{proposition}
\begin{proof}
Fix $x=(a,b,z)\in\mathcal{S}$ and write
\[
r_0=r(x),\;\;\;\;\alpha_0=\alpha(x).
\]
Let $V_1,V_2,\dots$ be the successive sampled vertices, with $V_0=b$, and define the transition times
\[
J_1=r_0,\;\;\;\;J_{n+1}=J_n+l(V_{n-1},V_n),\;\;\;\;n\geq1.
\]
The sampled vertices are measurable functions of the path, which is verified in Section \ref{sec_details}. Therefore, these variables remain available on any randomized extension. Since $\mu$ is atomless, with probability one the vertices $V_1,V_2,\dots$ are pairwise distinct and avoid two fixed vertices $a,b$. We work on this probability-one event. Put $Y=X(T)$.

First consider the initial edge. On the probability-one event just described,
\[
\{\gamma(Y)=(a,b)\}=\{T<J_1\}.
\]
The event on the left belongs to $\sigma(Y)$ and the event on the right belongs to $\sigma(T)$, since $J_1=r_0$ is deterministic. If $p=P(T<J_1)$, independence (of $T$ and $Y$) gives $p=p^2$. If $p=1$, then
$
T=\alpha(Y)-\alpha_0
$
almost surely, and Lemma \ref{lem_independence} implies that $T$ is constant. Hence, if $T$ is non-constant, then $T\geq J_1$ almost surely ($p=0$), under which
\[
\{\beta(Y)=b\}=\{J_1\leq T<J_2\}.
\]
On this event, $T=J_1+\alpha(Y)$. If the event has positive probability, Lemma \ref{lem_independence} again implies that $T$ is constant. Hence, if $T$ is non-constant, then $T\geq J_2$ almost surely.

Let $N\geq 2$ be the unique finite random integer satisfying
\[
J_N\leq T<J_{N+1},
\]
and define the completed vertex path
\[
\Pi_T=(V_0,\dots,V_{N-1}).
\]
It is a finite simple path, starts at $b$, ends at $\beta(Y)$, and satisfies
\begin{equation}
\label{eqn_T}
\begin{aligned}
    T=J_1+\sum_{k=1}^{N-1}l(V_{k-1},V_k)+\alpha(Y).
\end{aligned}
\end{equation}
The finite-path space is standard Borel, and $\Pi_T$ is measurable, which is verified in Section \ref{sec_details}.

Let $\rho=\mathrm{Law}(T)$ and $\lambda=\mathrm{Law}(Y)$. Take a regular conditional distribution of $(T,\Pi_T)$ given $Y$. For $\lambda$-almost every $y$, this conditional law is supported on simple paths that start at $b$, end at $\beta(y)$, and satisfy \eqref{eqn_T} with $Y=y$. On a new probability space, first sample $Y\sim\lambda$ and then, conditional on $Y$, sample two independent copies
\[
(T_1,\Pi_{T_1}),\;\;\;\;(T_2,\Pi_{T_2})
\]
of $(T,\Pi_T)$. Since $T$ is independent of $Y$, $D=T_1-T_2$ and $c(D)$ are independent of $Y$, where the latter uses the Borel measurability of $c$.

The two conditional copies have the same observed state $Y$. Thus both completed vertex paths start at $b$ and end at $\beta(Y)$. Subtracting \eqref{eqn_T} for the two copies cancels the deterministic term $J_1$ and the common current-edge contribution $\alpha(Y)$, giving
\[
D=L(\Pi_{T_1})-L(\Pi_{T_2}),
\]
where
\[
L(\pi)=\sum_{k=1}^ml(v_{k-1},v_k),\;\;\;\;\pi=(v_0,\dots,v_m).
\]
On $\{D\neq0\}$, Lemma \ref{lem_decode} therefore implies that $c(D)$ is a non-initial vertex occurring in both $\Pi_{T_1}$ and $\Pi_{T_2}$.
Let $\tilde{\Pi}_T$ be the random finite sequence of non-initial vertices of $\Pi_T$. For every fixed $v\in\mathcal{V}$,
\[
P(v\text{ occurs in }\tilde{\Pi}_T)\leq\sum_{k=1}^\infty P(V_k=v)=0.
\]
Conditional on $Y$, the corresponding sequences from $\Pi_{T_1}$ and $\Pi_{T_2}$ are independent copies of $\tilde{\Pi}_T$. Since $c(D)$ is independent of $Y$, Lemma \ref{lem_collision} shows that the probability that $c(D)$ occurs in both paths is zero. Hence 
\[
P(D\neq0)\leq P(c(D)\text{ occurs in both }\tilde{\Pi}_{T_1}\text{ and }\tilde{\Pi}_{T_2})=0,
\]
and therefore $D=0$ almost surely.

Finally, $T_1$ and $T_2$ are independent with common law $\rho$, while $T_1=T_2$ almost surely. Hence $\rho$ is a point mass, so $T$ is almost surely constant.
\end{proof}

\begin{corollary}
\label{cor_neq}
For every $x\in\mathcal{S}$ and every $0\leq s<t$,
\[
P_x(X(s)\in\cdot)\neq P_x(X(t)\in\cdot).
\]
\end{corollary}
\begin{proof}
Suppose instead that both laws are equal to a probability measure $\nu$. On an extension, let $B$ be an independent $\mathrm{Ber}(1/2)$ random variable. Set $T=s$ when $B=0$ and set $T=t$ when $B=1$. Then
\[
P_x(T\in\cdot, X(T)\in\cdot)=\frac{\delta_s+\delta_t}{2}\otimes \nu,
\]
so $T$ is independent of $X(T)$. Since $s\neq t$, $T$ is not almost surely constant, contradicting Proposition \ref{prop_rigidity}.
\end{proof}

We can now combine Proposition \ref{prop_rigidity} with Corollary \ref{cor_neq} to rule out wide-sense regeneration and complete the proof of Theorem \ref{thm_main}.
\begin{proof}[Proof of Theorem \ref{thm_main}]
Proposition 1 shows that $X$ is Harris recurrent.
Fix $x\in\mathcal{S}$ and suppose that randomized stopping times 
$
0\leq R_0<R_1<\cdots\infty
$
satisfy Definition \ref{def:regeneration} under $P_x$. Since $X(R_n)$ is the time-zero coordinate of $\Theta_{R_n}X$, condition (i) implies that $R_n$ is independent of $X(R_n)$. Proposition \ref{prop_rigidity} therefore implies that every $R_n$ is almost surely constant. Condition (ii) then implies that $P_x(X(R_n)\in\cdot)$ is independent of $n$, contradicting Corollary \ref{cor_neq}. Thus $X$ is not wide-sense regenerative.
\end{proof}

\section{A Rationally Independent Cantor Set}
\label{sec_cantor}
For completeness, we give an alternative direct construction of a Cantor set with the property required in Section \ref{sec_counter}. It is enough to exclude nontrivial integer relations, since rational relations can be cleared of denominators.

Enumerate all nonzero finite integer vectors, allowing repetitions, as
\[
q^{(1)},q^{(2)},\dots, \;\;\;\;q^{(j)}=(q_1^{(j)},\dots,q_{d_j}^{(j)})\in\mathbb{Z}^{d_j}\setminus\{0\},
\]
in such a way that $d_j\leq 2^j$ for every $j$ (this is easy to arrange since repetitions are allowed). We recursively construct, for each $n\geq0$, pairwise disjoint nondegenerate compact intervals
\[
\{I_\sigma:\sigma\in\{0,1\}^n\}\subset[1,2]
\]
such that
\begin{enumerate}[label=(\roman*)]
    \item $I_{\sigma0}$ and $I_{\sigma1}$ lie in the interior of $I_\sigma$;
    \item $\max_{|\sigma|=n}\mathrm{diam}(I_\sigma)\leq 2^{-n}$ for $n\geq1$;
    \item for every $j\leq n$ and every injection $\iota:\{1,\dots,d_j\}\rightarrow\{0,1\}^n$,
    \begin{equation}
    \label{eqn_cantor}
    \begin{aligned}
        \inf_{x_i\in I_{\iota(i)}}\abs{\sum_{i=1}^{d_j}q_i^{(j)}x_i}>0.
    \end{aligned}
    \end{equation}
\end{enumerate}
Start with $I_\emptyset=[1,2]$. Suppose level $n$ has been constructed. For each parent interval, choose two nonempty open subintervals with disjoint closures contained in its interior. A prospective center in each of the $2^{n+1}$ child intervals therefore ranges over a nonempty open box in $\R^{2^{n+1}}$. For each $j\leq n+1$ and each injection $\iota$ into the level-$(n+1)$ coordinates, the forbidden equality
\[
\sum_{i=1}^{d_j}q_i^{(j)}x_{\iota(i)}=0
\]
defines a proper hyperplane ($q^{(j)}\neq0$). There are only finitely many such hyperplanes, and their union cannot contain the open box. Choose the child centers outside this union. By continuity and finiteness of the constraints, sufficiently small compact intervals around these centers still satisfy all the strict inequalities in \eqref{eqn_cantor}; they may also be chosen pairwise disjoint, contained in the prescribed parent interiors, and of diameter at most $2^{-(n+1)}$. This completes the induction.

Set
\[
L_0=\bigcap_{n\geq0}\bigcup_{|\sigma|=n}I_\sigma.
\]
The binary nesting and the uniform decay of the diameters give a homeomorphism from $\{0,1\}^\mathbb{N}$ to $L_0$, so $L_0$ is a Cantor set. Suppose distinct $r_1,\dots,r_k\in L_0$ satisfy a nontrivial rational relation. After clearing denominators, let the resulting nonzero integer coefficient vector be $q^{(j)}$. At all sufficiently large levels, the points $r_1,\cdots,r_k$ lie in distinct basic intervals. Choosing such a level $n\geq j$ contradicts \eqref{eqn_cantor}. Hence $L_0$ is linearly independent over $\mathbb{Q}$.

\section{Technical Details}
\label{sec_details}
\subsection{Measurability in Lemma \ref{lem_decode}}
We now verify the Borel measurability of the function $c$ in Lemma \ref{lem_decode}. 
Let $\prec$ denote the lexicographic order on $\mathcal{V}=\{0,1\}^\mathbb{N}$. Thus $u\prec v$ if, at the first coordinate at which $u$ and $v$ differ, $u$ has value $0$ and $v$ has value $1$. This order is Borel, since
\[
\{(u,v):u\prec v\}=\bigcup_{k\geq1}\{(u,v):u_j=v_j\text{ for }j<k,\;u_k=0,v_k=1\},
\]
which is a countable union of Borel sets determined by finitely many coordinates.
Let
\[
\mathcal{E}=\{(u,v)\in\mathcal{V}^2:u\neq v\}
\]
be the space of (non-loop) directed edges. Only such edges occur in a simple path. As a Borel subset of $\mathcal{V}^2$, $\mathcal{E}$ is a standard Borel space\footnote{That is, there exists a bijection between $\mathcal E$ and a Borel subset of a Polish space such that both the bijection and its inverse are measurable.}. Equip $\mathcal{E}$ with the lexicographic order induced by $\prec$: for $e=(u,v)$ and $e'=(u',v')$, let
$e\prec_\mathcal{E}e'$ if $u\prec u'$, or if $u=u'$ and $v\prec v'$, which is again a Borel order. 

For $m,n\geq0$, let $\mathcal{D}_{m,n}$ consist of all tuples
\[
(e_1,\dots,e_m;f_1,\dots,f_n)
\]
such that
\[
e_1\prec_\mathcal{E}\cdots\prec_\mathcal{E} e_m,\;\;\;\;f_1\prec_\mathcal{E}\cdots\prec_\mathcal{E} f_n,
\]
and
\[
e_i\neq f_j,\;\;\;\;1\leq i\leq m,\;\;1\leq j\leq n.
\]
Each $\mathcal{D}_{m,n}$ is a Borel subset of $\mathcal{E}^{m+n}$. Hence the countable disjoint union
\[
\mathcal{D}=\bigcup_{m,n\geq0}\mathcal{D}_{m,n}
\]
is a standard Borel space. Every ordered pair $(A_+,A_-)$ of disjoint finite sets of directed edges has a unique representation in $\mathcal{D}$, obtained by listing the elements of each set in increasing order. Define $d:\mathcal{D}\rightarrow\R$ by
\[
d(A_+,A_-)=\sum_{e\in A_+}l(e)-\sum_{e\in A_-}l(e),
\]
as in the proof of Lemma \ref{lem_decode}, where it is shown to be one-to-one.
Moreover, the map $d$ is Borel, since on each $\mathcal{D}_{m,n}$ it is a finite sum of continuous functions.

Since $\mathcal{D}$ and $\R$ are standard Borel spaces, with $d:\mathcal{D}\rightarrow\R$ a one-to-one Borel map, Corollary 15.2 of \cite{kechris1995classical} implies that $G=d(\mathcal{D})$ is a Borel subset of $\R$, and $d$ is a Borel isomorphism between $\mathcal{D}$ and $G$. In particular, $d^{-1}:G\rightarrow\mathcal{D}$ is Borel measurable.

It remains only to verify that the lexicographic selection used in Lemma \ref{lem_decode} is Borel.
Let
\[
\bar{\pi}:\mathcal{E}\rightarrow\mathcal{V},\;\;\;\;\bar{\pi}(u,v)=v,
\]
denote the terminal vertex of an edge. For
\[
\delta=(e_1,\dots,e_m;f_1,\dots,f_n)\in\mathcal{D}_{m,n},
\]
define
\[
C(\delta)=\{\bar{\pi}(e_i):1\leq i\leq m\}\cap\{\bar{\pi}(f_j):1\leq j\leq n\}.
\]
Under the identification of $\delta$ with $(A_+,A_-)$, this is exactly the set $C(A_+,A_-)$ defined in Lemma \ref{lem_decode}. The set 
\[
\{\delta\in\mathcal{D}_{m,n}:C(\delta)\neq\emptyset\}
\]
is Borel, since it is determined by finitely many equalities of the form $\bar{\pi}(e_i)=\bar{\pi}(f_j)$. For each $1\leq i\leq m$, the condition that $\bar{\pi}(e_i)$ is the lexicographically smallest element of $C(\delta)$ is Borel. Indeed, it requires first that $\bar{\pi}(e_i)\in C(\delta)$, and then that, for $1\leq k\leq m$, either $\bar{\pi}(e_k)\notin C(\delta)$ or $\bar{\pi}(e_i)\preceq\bar{\pi}(e_k)$. Each of these conditions is Borel, since $\bar{\pi}$ is continuous and the lexicographic order $\preceq$ is Borel. As only finitely many indices are involved, their intersections and unions are Borel.

It follows that the map assigning to $\delta$ the lexicographically smallest element of $C(\delta)$, whenever $C(\delta)\neq\emptyset$, is Borel measurable. Indeed, for any Borel set $B\subset\mathcal{V}$, its inverse image is a finite union over $i$ of the Borel sets on which $\bar{\pi}(e_i)$ is the lexicographically smallest element of $C(\delta)$ and $\bar{\pi}(e_i)\in B$. Define $\kappa:\mathcal{D}\rightarrow\mathcal{V}$ to be this lexicographic minimum whenever $C(\delta)\neq\emptyset$, and otherwise some fixed element of $\mathcal{V}$ (as in the proof of Lemma \ref{lem_decode}). The preceding argument shows that $\kappa$ is Borel measurable on each $\mathcal{D}_{m,n}$ and hence on $\mathcal{D}$.

Finally, recall that $G=d(\mathcal{D})$ is Borel and that $d^{-1}:G\rightarrow\mathcal{D}$ is Borel measurable. For $r\in G$, the function $c$ in Lemma \ref{lem_decode} is therefore
\[
c(r)=\kappa(d^{-1}(r)).
\]
Thus $c$ is Borel measurable on $G$. On $\R\setminus G$, $c$ takes some fixed value in $\mathcal{V}$ (as in the proof of Lemma \ref{lem_decode}). Since $G$ is Borel, the resulting map $c:\R\rightarrow\mathcal{V}$ is Borel measurable.
\subsection{Measurability in Lemma \ref{lem_collision}}
We now verify the measurability of the set $\mathcal{A}$ in Lemma \ref{lem_collision}.
Adjoin an isolated cemetery point $\partial$ to $\mathcal{V}$ and equip $\mathcal{V}\cup\{\partial\}$ with any compatible Polish metric $d_0$. Let $K(y,\cdot)$ be a probability kernel from $\mathcal{S}$ to $\mathcal{V}\cup\{\partial\}$. Then
\[
(y,v)\mapsto K(y,\{v\})
\]
is Borel on $\mathcal{S}\times\mathcal{V}$, since
\[
K(y,\{v\})=\inf_{m\geq1}\int_{\mathcal{V}\cup\{\partial\}}I(d_0(w,v)<1/m)K(y,dw).
\]
Applying this to each conditional kernel $K_j$ in Lemma \ref{lem_collision} shows that
\[
\mathcal{A}=\bigcup_{j\geq1}\{(y,v):K_j(y,\{v\})>0\}
\]
is Borel.

\subsection{Measurability in Proposition \ref{prop_rigidity}}
We record two routine facts used in the proof of Proposition \ref{prop_rigidity}. First, the sampled vertices are Borel functions of the path. Fix $x=(a,b,z)$ and let $J_1=r(x)$. Then $X(J_1)=(b,V_1,0_\mathbb{T})$, so $V_1$ is the second vertex coordinate of $X(J_1)$. Recursively, once $J_n,V_{n-1},V_n$ are known, $J_{n+1}=J_n+l(V_{n-1},V_n)$ and $V_{n+1}$ is the second vertex coordinate of $X(J_{n+1})$. The joint evaluation map $(\omega,t)\mapsto\omega(t)$ on Skorokhod space is Borel, so this recursion makes every $V_n$ (and $J_n$) a Borel function of the entire path.

Second, let
\[
\mathcal{P}=\bigcup_{m\geq1}\mathcal{V}^m
\]
be the space of nonempty finite vertex sequences, which is a standard Borel space.
When $T\geq J_2$ almost surely, as in the proof of Proposition \ref{prop_rigidity}, the index $N\geq2$ determined by $J_N\leq T<J_{N+1}$ is measurable, and hence so is
\[
\Pi_T=(V_0,\dots,V_{N-1})\in\mathcal{P}.
\]
Thus the regular conditional distribution of $(T,\Pi_T)$ given $X(T)$ used in the proof of Proposition \ref{prop_rigidity} exists.

\subsection{Strong Markov property of the counterexample}

For $x=(a,b,z)\in\mathcal{S}$, construct the process on the product
space carrying
\[
V_1,V_2,\dots\stackrel{\mathrm{iid}}{\sim}\mu,
\qquad V_0=b,
\]
and put
\[
J_0=0,\qquad
J_1=r(x),\qquad
J_{n+1}=J_n+l(V_{n-1},V_n),\qquad n\geq1.
\]
Then $J_n\rightarrow\infty$ almost surely, and the path is determined
by $x$ and the sequence $(V_n)$. For each fixed $t$, the map from
$(x,(V_n))$ to $X(t)$ is Borel: the transition times are Borel and the
explicit path formula applies on the countable Borel partition
\[
\{J_n\leq t<J_{n+1}\},\qquad n\geq0.
\]
Since the Borel $\sigma$-field of the Skorokhod space
$D([0,\infty),\mathcal{S})$ is generated by evaluations at rational
times, the full path-construction map is Borel. Integrating over the iid
marks therefore makes $x\mapsto P_x$ a Borel probability kernel on path
space.

Fix $x\in\mathcal{S}$ for the remainder of the proof. Let
$\mathcal{N}_x$ denote the collection of all subsets of $P_x$-null
events in the underlying product space. Define
\[
\mathcal{H}_0
=
\sigma\bigl(\{\emptyset,\Omega\}\cup\mathcal{N}_x\bigr),
\qquad
\mathcal{H}_n
=
\sigma\bigl(\sigma(V_1,\dots,V_n)\cup\mathcal{N}_x\bigr),
\qquad n\geq1.
\]
Thus $(\mathcal{H}_n)$ is the completed filtration of the iid mark
sequence, using the same collection of null sets at every index. For an
integer-valued $(\mathcal{H}_n)$-stopping time $M$, let
\[
\mathcal{H}_M
=
\{A:A\cap\{M=n\}\in\mathcal{H}_n
\text{ for every }n\geq0\}.
\]
For $t\geq0$, define
\[
N_t=\max\{n\geq0:J_n\leq t\},
\qquad
\mathcal{G}_t=\mathcal{H}_{N_t}.
\]
The maximum is finite because $J_n\rightarrow\infty$. Moreover, $N_t$
is an $(\mathcal{H}_n)$-stopping time, since
\[
\{N_t\leq n\}=\{t<J_{n+1}\}\in\mathcal{H}_n.
\]
The family $(\mathcal{G}_t)$ is an increasing filtration. Indeed,
$N_s\leq N_t$ when $s\leq t$, while
$\mathcal{H}_M\subseteq\mathcal{H}_{M'}$ for integer-valued stopping
times $M\leq M'$. The process $X$ is adapted to this filtration: on
$\{N_t=n\}$, the state $X(t)$ is a Borel function of
$V_1,\dots,V_n$ and of the fixed initial state $x$. In particular, the
raw natural filtration
\[
\mathcal{F}_t^0=\sigma(X(s):0\leq s\leq t)
\]
satisfies
$
\mathcal{F}_t^0\subseteq\mathcal{G}_t.
$
By construction, each $\mathcal{G}_t$ contains $\mathcal{N}_x$.

The filtration $(\mathcal{G}_t)$ is also right-continuous. Indeed, let
$
A\in\bigcap_{m\geq1}\mathcal{G}_{t+1/m}.
$
For every $n\geq0$,
\[
A\cap\{N_t=n\}
=
\bigcup_{m\geq1}
A\cap\{N_t=n,\ J_{n+1}>t+1/m\}.
\]
On the event in the $m$-th term, $N_{t+1/m}=n$. Hence that term equals
\[
\bigl(A\cap\{N_{t+1/m}=n\}\bigr)
\cap
\{N_t=n,\ J_{n+1}>t+1/m\},
\]
which belongs to $\mathcal{H}_n$. Thus
$
A\cap\{N_t=n\}\in\mathcal{H}_n
$
for every $n$, proving $A\in\mathcal{G}_t$. The reverse inclusion
follows from the increasing property.

Now let $\tau<\infty$ be a $(\mathcal{G}_t)$-stopping time and set
\[
N=N_\tau=\max\{n\geq0:J_n\leq\tau\}.
\]
We first note that $N$ is an $(\mathcal{H}_n)$-stopping time. For every
$n$,
\[
\{N\leq n\}
=
\{\tau<J_{n+1}\}
=
\bigcup_{q\in\mathbb{Q}\cap[0,\infty)}
\{\tau\leq q\}\cap\{q<J_{n+1}\}.
\]
For fixed $q$, the event $\{\tau\leq q\}$ belongs to
$
\mathcal{G}_q=\mathcal{H}_{N_q},
$
while
$
\{q<J_{n+1}\}=\{N_q\leq n\}.
$
By the definition of the stopped sigma-field $\mathcal{H}_{N_q}$,
their intersection belongs to $\mathcal{H}_n$.

Let
\[
\mathcal{G}_\tau
=
\{A:A\cap\{\tau\leq t\}\in\mathcal{G}_t
\text{ for every }t\geq0\}
\]
denote the stopped sigma-field at $\tau$. We claim that
\begin{equation}
\label{eqn_GH}
\mathcal{G}_\tau\subseteq\mathcal{H}_N.
\end{equation}
Let
\[
\tau_m=2^{-m}\bigl(\lfloor2^m\tau\rfloor+1\bigr),
\]
so that $\tau_m\downarrow\tau$. Each $\tau_m$ is a countable-range
$(\mathcal{G}_t)$-stopping time; for every dyadic grid point
$q=k2^{-m}$,
\[
\{\tau_m\leq q\}=\{\tau<q\}\in\mathcal{G}_q.
\]
If $A\in\mathcal{G}_\tau$, then
$A\in\mathcal{G}_{\tau_m}$ because $\tau\leq\tau_m$. For every $n$,
\[
A\cap\{N=n\}
=
\bigcup_{m\geq1}
A\cap\{J_n\leq\tau,\ \tau_m<J_{n+1}\}.
\]
For $n\geq1$,
\[
\{\tau<J_n\}
=
\bigcup_{q\in\mathbb{Q}\cap[0,\infty)}
\{\tau\leq q\}\cap\{q<J_n\}
\in\mathcal{H}_{n-1}.
\]
Thus
\[
\{J_n\leq\tau\}\in\mathcal{H}_{n-1}
\subseteq\mathcal{H}_n;
\]
the case $n=0$ is immediate. Furthermore, because $\tau_m$ has
countable range,
\[
A\cap\{\tau_m<J_{n+1}\}
=
\bigcup_{q\in2^{-m}\mathbb{N}}
A\cap\{\tau_m=q\}\cap\{q<J_{n+1}\}.
\]
Here
$
A\cap\{\tau_m=q\}\in\mathcal{G}_q
=\mathcal{H}_{N_q},
$
and intersecting with
$
\{q<J_{n+1}\}=\{N_q\leq n\}
$
gives an event in $\mathcal{H}_n$. Thus every term in the preceding
union belongs to $\mathcal{H}_n$, and consequently
$
A\cap\{N=n\}\in\mathcal{H}_n
$
for every $n$. This proves \eqref{eqn_GH}.

The unused tail
$
(V_{N+1},V_{N+2},\dots)
$
has law $\mu^{\otimes\mathbb{N}}$ and is independent of
$\mathcal{H}_N$. Indeed, independence of a deterministic-index iid tail
from $\sigma(V_1,\dots,V_n)$ extends to its completion
$\mathcal{H}_n$. Hence, for $A\in\mathcal{H}_N$ and every bounded
Borel function $\varphi$ on $\mathcal{V}^{\mathbb{N}}$,
\[
\begin{aligned}
\E_x\left[
I_A\varphi(V_{N+1},V_{N+2},\dots)
\right]=&
\sum_{n=0}^{\infty}
\E_x\left[
I_{A\cap\{N=n\}}
\varphi(V_{n+1},V_{n+2},\dots)
\right]\\
=&
\sum_{n=0}^{\infty}
P_x(A\cap\{N=n\})
\int\varphi\,d\mu^{\otimes\mathbb{N}}\\
\quad=&
P_x(A)\int\varphi\,d\mu^{\otimes\mathbb{N}}.
\end{aligned}
\]
By \eqref{eqn_GH}, the unused tail is therefore independent of
$\mathcal{G}_\tau$.

Let $\Psi(y,w)$ denote the path obtained by starting from state $y$ and
using the fresh vertex sequence $w$. The Borel measurability of the
path-construction map established above implies that $\Psi$ is Borel.
Pathwise,
\[
\Theta_\tau X
=
\Psi\bigl(
X(\tau),(V_{N+1},V_{N+2},\dots)
\bigr).
\]
Since $X(\tau)$ is $\mathcal{G}_\tau$-measurable\footnote{Indeed, $X$
is adapted to $(\mathcal{G}_t)$ and has right-continuous paths, hence is
progressively measurable; therefore $X(\tau)$ is
$\mathcal{G}_\tau$-measurable.} and the unused tail is independent of
$\mathcal{G}_\tau$ with law $\mu^{\otimes\mathbb{N}}$, for every
bounded Borel function $F$ on path space,
\begin{equation}
\label{eqn_tau}
\E_x\left[
F(\Theta_\tau X)\mid\mathcal{G}_\tau
\right]
=
g(X(\tau))
\qquad P_x\text{-a.s.,}
\end{equation}
where
$
g(y)=\E_y[F(X)]
$
is Borel because $y\mapsto P_y$ is a Borel probability kernel.

It remains only to pass to the usual augmented natural filtration. Let
\[
\mathcal{F}_\infty^0
=
\sigma\left(\bigcup_{t\geq0}\mathcal{F}_t^0\right),
\]
and let $\mathcal{N}_x^X$ be the collection of all subsets of
$P_x$-null events in $\mathcal{F}_\infty^0$. The usual augmentation of
the natural filtration is
\[
\mathcal{F}_t
=
\bigcap_{s>t}
\sigma\bigl(\mathcal{F}_s^0\cup\mathcal{N}_x^X\bigr).
\]
Since
$
\mathcal{F}_s^0\subseteq\mathcal{G}_s
$, $
\mathcal{N}_x^X\subseteq\mathcal{N}_x\subseteq\mathcal{G}_s
$,
and $(\mathcal{G}_t)$ is right-continuous,
\[
\mathcal{F}_t
\subseteq
\bigcap_{s>t}\mathcal{G}_s
=
\mathcal{G}_t.
\]
Thus every finite $(\mathcal{F}_t)$-stopping time is also a finite
$(\mathcal{G}_t)$-stopping time, with
$
\mathcal{F}_\tau\subseteq\mathcal{G}_\tau.
$
Taking conditional expectations in \eqref{eqn_tau} with respect to
$\mathcal{F}_\tau$ gives
\[
\E_x\left[
F(\Theta_\tau X)\mid\mathcal{F}_\tau
\right]
=
\E_x\left[
g(X(\tau))\mid\mathcal{F}_\tau
\right]
=
g(X(\tau)),
\]
where the last equality follows because $X(\tau)$ is
$\mathcal{F}_\tau$-measurable. Therefore, $X$ is strong Markov with
respect to its usual augmented natural filtration.
\begin{appendix}
\end{appendix}
\bibliographystyle{imsart-nameyear} 
\bibliography{bibliography}       

\begin{thebibliography}{8}

\bibitem[\protect\citeauthoryear{Athreya and Ney}{1978}]{athreya1978new}
\begin{barticle}[author]
\bauthor{\bsnm{Athreya},~\bfnm{Krishna~B}\binits{K.~B.}} \AND \bauthor{\bsnm{Ney},~\bfnm{Peter}\binits{P.}}
(\byear{1978}).
\btitle{A new approach to the limit theory of recurrent Markov chains}.
\bjournal{Transactions of the American Mathematical Society}
\bvolume{245}
\bpages{493--501}.
\end{barticle}
\endbibitem

\bibitem[\protect\citeauthoryear{Glynn}{2011}]{glynn2011wide}
\begin{barticle}[author]
\bauthor{\bsnm{Glynn},~\bfnm{Peter~W}\binits{P.~W.}}
(\byear{2011}).
\btitle{Wide-sense regeneration for Harris recurrent Markov processes: an open problem}.
\bjournal{Queueing Systems}
\bvolume{68}
\bpages{305--311}.
\end{barticle}
\endbibitem

\bibitem[\protect\citeauthoryear{Harris}{1956}]{harris1956existence}
\begin{binproceedings}[author]
\bauthor{\bsnm{Harris},~\bfnm{Theodore~E}\binits{T.~E.}}
(\byear{1956}).
\btitle{The existence of stationary measures for certain Markov processes}.
In \bbooktitle{Proceedings of the Third Berkeley Symposium on Mathematical Statistics and Probability}
\bvolume{2}
\bpages{113--124}.
\end{binproceedings}
\endbibitem

\bibitem[\protect\citeauthoryear{Kechris}{1995}]{kechris1995classical}
\begin{bbook}[author]
\bauthor{\bsnm{Kechris},~\bfnm{Alexander~S.}\binits{A.~S.}}
(\byear{1995}).
\btitle{Classical Descriptive Set Theory}.
\bseries{Graduate Texts in Mathematics}
\bvolume{156}.
\bpublisher{Springer-Verlag}, \baddress{New York}.
\end{bbook}
\endbibitem

\bibitem[\protect\citeauthoryear{Meyn and Tweedie}{2012}]{meyn2012markov}
\begin{bbook}[author]
\bauthor{\bsnm{Meyn},~\bfnm{Sean~P}\binits{S.~P.}} \AND \bauthor{\bsnm{Tweedie},~\bfnm{Richard~L}\binits{R.~L.}}
(\byear{2012}).
\btitle{Markov chains and stochastic stability}.
\bpublisher{Springer Science \& Business Media}.
\end{bbook}
\endbibitem

\bibitem[\protect\citeauthoryear{Mycielski}{1964}]{mycielski1964independent}
\begin{barticle}[author]
\bauthor{\bsnm{Mycielski},~\bfnm{Jan}\binits{J.}}
(\byear{1964}).
\btitle{Independent sets in topological algebras}.
\bjournal{Fundamenta Mathematicae}
\bvolume{55}
\bpages{139--147}.
\end{barticle}
\endbibitem

\bibitem[\protect\citeauthoryear{Nummelin}{1978}]{nummelin1978splitting}
\begin{barticle}[author]
\bauthor{\bsnm{Nummelin},~\bfnm{Esa}\binits{E.}}
(\byear{1978}).
\btitle{A splitting technique for Harris recurrent Markov chains}.
\bjournal{Zeitschrift f{\"u}r Wahrscheinlichkeitstheorie und verwandte Gebiete}
\bvolume{43}
\bpages{309--318}.
\end{barticle}
\endbibitem

\bibitem[\protect\citeauthoryear{von Neumann}{1928}]{neumann1928system}
\begin{barticle}[author]
\bauthor{\bparticle{von} \bsnm{Neumann},~\bfnm{John}\binits{J.}}
(\byear{1928}).
\btitle{Ein System algebraisch unabh{\"a}ngiger zahlen}.
\bjournal{Mathematische Annalen}
\bvolume{99}
\bpages{134--141}.
\end{barticle}
\endbibitem

\end{thebibliography}


\end{document}